\documentclass{amsart}
\usepackage[T2A]{fontenc}
\usepackage[russian,english]{babel}
\newtheorem{теор}{Тheorem}[section]
\newtheorem{лем}[теор]{Lemma}
\newtheorem{зам}[теор]{Remark}

\newtheorem{анн}[теор]{Annotation}
\newtheorem{опр}[теор]{Definition}
\newtheorem{ключ}[теор]{ Keywords.}

\newtheorem{thm}{Theorem}[section]

\begin{document}

\begin{center}
\textbf{{\large {\ 
Coefficient Identification Problem with Integral Overdetermination Condition for Diffusion Equations}}}\\[0pt]
\medskip \textbf{R.R. Ashurov$^{1,2}$ and O.T. Mukhiddinova$^{3,1}$}\\[0pt]
\textit{ashurovr@gmail.com, oqila1992@mail.ru \\[0pt]}

\smallskip
\textit{{\it $^1$V.I. Romanovskiy Institute of Mathematics, Uzbekistan Academy of Science, University str.,9, Olmazor district, Tashkent, 100174, Uzbekistan;\\
$^2$ Central Asian University, 264 Milliy Bog Street, Barkamol MFY, Mirzo Ulugbek District, Tashkent 111221, Uzbekistan; \\
$^3$ Tashkent University of Information Technologies named after Muhammad al-Khwarizmi,
108 Amir Temur Avenue,Tashkent, 100200, Uzbekistan; }}\\

\end{center}

\begin{анн}
In this paper, we investigate a nonlinear inverse problem aimed at recovering a coefficient $a(t, x)$, dependent on both time and a subset of spatial variables, in a diffusion equation \( u_t - \Delta_x u -  u_{yy} +a(t, x) u = f(t,x,y) \), using an additional measurement given as an integral over the spatial domain. Here \(x \in G \subset \mathbb{R}^m\) and \(y \in (0, \pi)\). We establish theorems on the existence and uniqueness of both local and global weak solutions. Furthermore, we demonstrate that, under sufficient smoothness of the problem data, there exists a uniquely determined strong solution (both local and global) to the inverse problem. Our approach combines the Fourier method with a priori estimates. Previous studies have addressed similar inverse problems for parabolic equations defined over the entire space.
\end{анн}

\begin{ключ}
Diffusion equation, inverse problem of coefficient identification,  uniqueness and existence of solution, weak, strong, local and global solutions, Fourier method.
\end{ключ}
\section{Introduction}

Let $ Q=(0,T) \times G\times  (0, \pi)$
and
$ D= G \times (0, \pi)$, where $G\subset \mathbb{R}^m$, $m\leq 3$, is a bounded domain, with a sufficiently smooth boundary. Let $G_T = (0,T) \times G$.
Consider the following initial - boundary value problem
\begin{equation} \label{1}
	\begin{cases}
		 u_t - \Delta_x u -  u_{yy} +a(t, x) u = f(t,x,y) , \quad (t,x,y) \in Q, \\
        
         u(0,x,y) = \varphi(x,y), \quad (x,y) \in D, \\
         
         u(t,x,y) =  0,\quad t\in [0,T], \quad (x, y)\in \partial D.
	\end{cases}
\end{equation} 
 Here $\Delta_x$ is the Laplace operator in variables $x$, \( f \) and \( \varphi \) are given functions $u_t:=D_tu =\frac{\partial u}{\partial t}$. In what follows we will denote the operators $\Delta_x$ and $\nabla_x$ as $\Delta$ and $\nabla$; it is clear that this will not cause confusion.
 
In equation (\ref{1}), the coefficient \( a = a(t,x) \) is assumed to depend on time \( t \) and the spatial variables \( x \), but not on the spatial component \( y \). This reflects a physical property that may vary with time and the spatial coordinates \( x \), while remaining independent of the depth variable \( y \), for example. If the function \(a(t, x)\) is given, then problem (\ref{1}) is called \textit{the forward problem} and its solution exists and is unique under certain conditions on the problem’s data (see, e.g., \cite{Ladijenskaya}, Chapter 3,  \cite{AMux2}, \cite{AMux}).

Now we will assume that the coefficient \( a(t, x) \) of $u$ is unknown and must be determined. The main goal of this paper is to study \textit{the inverse problem }of determining a pair of functions \( \{ u(t, x, y), a(t, x) \} \) under the overdetermination condition specified by the following integral measurement:
\begin{equation} \label{overdetermination}
\int_0^\pi u(t, x, y) \omega(y) dy = \psi(t, x), \,\, \psi(t, x)\neq 0,\quad t \in [0, T], \quad x \in G,
\end{equation}
where $\omega$ and \( \psi \) are known functions, the conditions on which we will determine later.

Let us mention that the recovery of time-dependent and time and spatial-dependent parameters appearing in the diffusion
equation (\ref{1}) have been considered by several authors (see, for example, \cite{Kaban}). The problem of coefficient recovery is related to ill-posed problems for equations of mathematical physics, the main methods for solving which were developed in the works of A. N. Tikhonov, M. M. Lavrentiev, A. I. Prilepko, V. G. Romanov (see, for example, \cite{1.1}).

Let us dwell only on some works that consider problems close to our inverse problem (\ref{1}) - (\ref{overdetermination}). A series of relatively old works by N. Ya. Beznoschenko \cite{4} - \cite{6} is devoted to the study of a similar problem, where the spatial region consists of the unbounded domain \( \mathbb{R}^{m+1}_+ = \{ (x, y) \in \mathbb{R}^{m+1} : y > 0 \} \). The key objective is to establish sufficient conditions for the \textit{global existence} of a solution to the inverse problem of determining the unknown coefficient \( a(x,t) \) in equation (\ref{1}) with the overdetermination condition \(u(t, x, 0) = \psi (t, x)\). The methodology is based on reducing the inverse problem to a well-posed direct problem, establishing a priori estimates and fixed-point theorems. 

In the paper \cite{7} the authors study a class of inverse problems, in particular, on the restoration of the coefficient $a(t, x)$ in the parabolic equation (\ref{1}), paying special attention to the determination of unknown functions that do not depend on one spatial variable. Taking the value of the solution to the forward problem at $y=0$ as the overdetermination condition, the authors managed to prove the stability of the solution to the inverse problem. To prove this result, the authors established new the Carleman estimates adapted to parabolic equations, where the unknown coefficient does not depend on one spatial component. The key innovation is the construction of the Carleman weight functions that effectively take into account spatial invariance, allowing precise control of the analysis errors.

The authors of \cite{10} investigated a similar inverse problem using a different method and proved the stability of the solution. Namely, the authors first studied the inverse problem of determining a source dependent on both time and space variables for fractional and classical diffusion equations in a cylindrical domain from boundary dimensions. Under suitable boundary conditions, they proved that a certain class of sources independent of one spatial direction can be reconstructed from boundary dimensions. The authors also establish some results on Lipschitz stability for source reconstruction, which they then apply to prove the stability of the coefficient $a(t, x)$.

In the paper \cite{9} the inverse problem of identifying coefficients for a linear parabolic partial differential equation is considered. The main attention is paid to the simultaneous determination of three unknown coefficients: the coefficient of the time derivative and two lower coefficients. The unknown functions depend on \((t, x)\), and the equation is considered in a two-dimensional space $x\in \mathbb{R}$, $y\in \mathbb{R}$. In proving the theorem of the existence and uniqueness of the solution to the inverse problem, a method is used that allows, using overdetermination conditions, to reduce the original inverse problem to a direct problem for a loaded equation (containing traces of unknown functions and their derivatives). The correctness of the resulting direct problem is studied using the weak approximation method.

In conclusion, we note a very interesting work by A. I. Kozhanov \cite{8}, where the inverse problem of restoring the potential $a(t, x, y)$ which depends on both $t$ and all spatial variables is considered. The overdetermination condition has the form $u(T, x, y) = \psi(x, y)$. There is an interesting additional condition $a_t(t, x, y) + \mu (t, x, y) a (t, x, y) =0$. Note that if $\mu =0$, then this condition means that $a$ does not depend on $t$. For proof, the indicated problems are reduced to nonlocal boundary value problems for nonlinear equations of composite type; the existence of a solution to the latter (which is proved in the work) gives the existence of a solution to the original inverse problems.

This work consists of eight sections. The next section is auxiliary, which presents the Poincaré inequalities, Sobolev embedding theorems, and the definition of the degree of a self-adjoint positive operator. In Section 3, we define a weak solution to the inverse problem and formulate the main results. The solution to the inverse problem is sought in the form of the Fourier series with unknown coefficients, and in Section 4, an infinite system of integro-differential equations is derived to determine these coefficients and a priori estimates are established. Sections 5 and 6 prove the main results of the work. Section 7 is devoted to the study of a strong solution to the problem. Finally, Section 8 provides a conclusion.

\section{Preliminaries}

\subsection{The Poincaré type inequality}
First we  prove the Poincaré type inequality.
\begin{лем}\label{Poincaré} Let a function \( g \) defined on \([0, T]\) belongs to $W_2^1(0, T)$. Then
\[
\int_0^T g(s)^2 \, ds \leq T^2 \int_0^T (g'(s))^2 ds + 2\, T\, g^2 (0).
\]

\end{лем}
\begin{proof}Note that \( g \in C[0, T] \) and $g'$ is an integrable function. Then we can write:
   \[
   g(t) = \int_0^t g'(s) \, ds  + g(0).
   \]
   Consider the square of \( g(x) \) (note that $(a+b)^2 \leq  2\, a^2 + 2 \, b^2$):
   \[
   g(t)^2 \leq 2\, \left( \int_0^t g'(s) \, ds \right)^2 +2 \, g^2 (0).
   \]
   By the Cauchy-Schwarz inequality for integrals:
   \[
   \left( \int_0^t g'(s) \cdot 1 \, ds \right)^2 \leq \left( \int_0^t (g'(s))^2 \, ds \right) \left( \int_0^t 1^2 \, ds \right).
   \]
   Then   \[
   g(t)^2 \leq 2\, \left( \int_0^t (g'(s))^2 \, ds \right) \cdot t  +2 \, g^2 (0).
   \]
   Note that \( \int_0^t (g'(s))^2 \, ds \leq \int_0^T (g'(s))^2 \, ds \), since \( t \leq T \), so:
   \[
   g(t)^2 \leq 2\,  t \int_0^T (g'(s))^2 \, ds  +2 \, g^2(0).
   \]
Integrate both sides of this inequality  over \([0, T]\):
   \[
   \int_0^T g(t)^2 \, dt \leq T^2\, \int_0^T (g'(t))^2 \, dt + 2\,  T g^2 (0).
  \]
\end{proof}

\subsection{Poincaré Inequality with Dirichlet Boundary Conditions}Let \( W_2^\ell (\Omega) \) denote the Sobolev space, where $\Omega\subset \mathbb{R}^N$, $N\geq 1$, is a bounded domain with the smooth boundary, $\ell$ is a positive real number. Then, the symbol \( \dot{W}_2^\ell(\Omega) \) represents the closure of the set \( C_0^\infty(\Omega) \) with respect to the norm of \( W_2^\ell(\Omega) \). Let \( B \) be a Banach space. We denote by \( L_2(0, T; B) \) the space of functions that are in \( L_2(0, T) \) and take values in \( B \). The spaces \( C(0, T; B) \) and \( W_2^\ell(\Omega; B) \) are defined similarly. 

\begin{зам}\label{norms}Let $G$ and $G_T$ are the domains as defined above. The norm in the space \( W_2^\ell(G; \dot{W}_2^\tau(0, \pi)) \) is denoted by \( \|\cdot\|_{\ell, \tau} \) and in the space  \( L_2(G_T; \dot{W}_2^\tau(0, \pi)) \) is denoted by \( \|\cdot\|_{G_T, \tau} \).
\end{зам}

For functions \( v \in \dot{W}_2^1(\Omega) \), i.e., functions in the Sobolev space with zero trace on the boundary \( \partial \Omega \), the Poincaré inequality takes the form (see, e.g, \cite{Evans}, p. 289):
\begin{equation}\label{DirixlePoincaré}
\int_{\Omega} |v(x)|^2 \, dx \leq C_P \int_{\Omega} |\nabla v(x)|^2 \, dx.    
\end{equation}

The optimal constant \( C_P \) in this inequality is determined as the reciprocal of the first eigenvalue of the Laplace operator \( -\Delta \) with Dirichlet boundary conditions:
\[
-\Delta v = \lambda v \quad \text{in } \Omega, \quad v = 0 \text{ on } \partial \Omega.
\]
The first eigenvalue \( \lambda_1 \) is found as the minimum of the Rayleigh quotient (see \cite{Evans}, p. 342):
\[
\lambda_1 = \inf_{v \in \dot{W}_2^1(\Omega), \, v \neq 0} \frac{\int_{\Omega} |\nabla v(x)|^2 \, dx}{\int_{\Omega} |v(x)|^2 \, dx}.
\]
Thus, the optimal constant is:
\[
C_P = \frac{1}{\lambda_1}.
\]
If the measure \( |\Omega| \) becomes very small, this corresponds to a reduction in the scale of the domain, which increases \( \lambda_1 \). Indeed, for an arbitrary bounded domain \( \Omega \) with a smooth boundary, the first eigenvalue \( \lambda_1 \) scales inversely proportional to the square of the characteristic size of the domain.

If the measure \( |\Omega| \to 0 \), this is equivalent to shrinking the domain. For a domain \( \Omega_\epsilon \), obtained from \( \Omega \) by a homothety with coefficient \( \epsilon \), the measure is:
\[
|\Omega_\epsilon| = \epsilon^N |\Omega|,
\]
and the eigenvalue scales as:
\[
\lambda_1(\Omega_\epsilon) = \frac{\lambda_1(\Omega)}{\epsilon^2}.
\]
Then:
\[
C_P(\Omega_\epsilon) = \frac{1}{\lambda_1(\Omega_\epsilon)} = \frac{\epsilon^2}{\lambda_1(\Omega)}.
\]
Since \( |\Omega_\epsilon| = \epsilon^N |\Omega| \), we have \( \epsilon = \left( \frac{|\Omega_\epsilon|}{|\Omega|} \right)^{1/N} \), and:
\[
C_P(\Omega_\epsilon) = \frac{1}{\lambda_1(\Omega)} \left( \frac{|\Omega_\epsilon|}{|\Omega|} \right)^{2/N}.
\]
As \( |\Omega_\epsilon| \to 0 \), \( \epsilon \to 0 \), and:
\[
C_P(\Omega_\epsilon) \to 0.
\]
\begin{зам}\label{optimalC}
    Using the Rayleigh quotient, it can be shown that if \( |\Omega| \to 0 \), then \( \lambda_1 \to \infty \) and \( C_P \to 0 \). Thus, the constant \( C_P \) can be made arbitrarily small if the measure of the region \( |\Omega| \) tends to zero.
\end{зам}
\subsection{Sobolev's theorem}
    The Sobolev embedding theorem states that for a bounded domain $ \Omega \subset \mathbb{R}^N $ with a sufficiently smooth boundary (e.g., Lipschitz or smoother), the Sobolev space $ W_p^k(G) $ embeds continuously into $ L_q(\Omega) $, i.e., $ W_p^k(\Omega) \hookrightarrow L_q(\Omega) $, provided the condition (see \cite{Brezis}, p. 278):
$$\frac{1}{q} \geq \frac{1}{p} - \frac{k}{N}$$
is satisfied. So the Sobolev inequality 
\begin{equation}\label{Sobolev}
    \|v\|_{L_4(G_T)} \leq C_S \|v\|_{W_2^1(G_T)} 
 \end{equation}
 holds for an \( m \)-dimensional bounded domain \( G \) with a smooth boundary if:
\[
m \leq 3,
\]
with the constant $C_S$ depending on the geometry of the region $G$.
\begin{зам}\label{m}
The restriction on the dimension of the domain $G$ arose precisely in connection with the application of this estimate.
\end{зам}
\subsection{Operator degrees and the Cauchy inequality}
Let us denote by $L_0$ an operator $-Y''(y)$ with domain $D(L_0) = \{Y\in C^2(0, \pi) \cap C[0, \pi]: Y(0)= Y(\pi)=0\}$. The set of eigenfunctions of this operator has the form $\{ \sin k y\}$ and eigenvalues has the form $\lambda_k = k^2$. Let $L$ be a self-adjoint extension in $L_2(0, \pi)$ of $L_0$.  Then $L$ is a positive operator with domain $D(L) = W_2^2(0, \pi) \cap \dot{W}_2^1(0, \pi) $ (see, for example, \cite{Berezans’kii}, Chapter 2). Therefore, we can define fractional powers of this operator using the von Neumann theorem. Namely, let $\tau$ be an arbitrary non-negative number. Then
\[
L^\tau p(y) = \sum_{k=1}^\infty \lambda_k^\tau\,  p_k\,  \sin k y, \,\, D(L^\tau) = \{p\in L_2(\Omega):  \sum_{k=1}^\infty \lambda_k^{2\tau}\, |p_k|^2 < \infty\}.
\]
Here after, the symbol \( p_k \) denotes the Fourier coefficients of a function \( p(y) \in L_2(0, \pi) \) with respect to the system \( \{\sin k y\} \), defined as:
\begin{equation}\label{CoefFur}
    p_k = \frac{2}{\pi} \int_0^\pi v(y) \sin \sqrt{\lambda_k} y \, dy.
\end{equation}
The domain of definition $D(L^\tau)$ is well studied in the work D. Fujiwara \cite{Fujiwara}. In particular, it is proved:
$$
  D(L^\tau) = \dot{W}_2^{2\tau}(0, \pi), \,\, \text{if}\,\, \,\,\frac{1}{4}< \tau < \frac{1}{2},\,\, \text{or}, \,\, \frac{1}{2}< \tau < \frac{3}{4}.
$$

In conclusion, we note that we will often use the Cauchy inequality
with the constant $\delta>0$: 
\begin{equation}\label{Cauchy inequality}
    ab \leq \frac{\delta}{2} a^2 + \frac{1}{2 \delta} b^2, \,\ a, b >0.
    \end{equation}

\section{Definition of the Weak Solution and Formulation of the Main Result}

First, we find a formal representation for the unknown function $a$ using the overdetermination condition \eqref{overdetermination}. To do this, we multiply the equation in (\ref{1}) by the function $\omega$ and integrate over the domain $(0, \pi)$. Then
\begin{equation*}
   \psi_t - \Delta_x \psi - (u_{yy}, \omega) +a \psi  = (f(t, x, \cdot), \omega) .
\end{equation*}
Here and below, the symbol $(\cdot, \cdot)$ will denote the scalar product in $L_2(0, \pi)$, and the norm of this space will be denoted by $||\cdot||$.

Now we need the following conditions for the function $\omega$:
$$
    \omega \in \dot{W}^1_2(0, \pi)  \cap W^2_2(0, \pi)\,\,\text{and} \,\int_0^\pi \varphi(x,y) \omega(y) dy= \psi(0, x).
$$
Taking into account this property, we get $( u_{yy}, \omega) =  ( u, \omega'')$. 

Note, that \( \psi(t, x) \neq 0 \) for all \( (t, x) \in [0, T] \times G \).  Therefore the desired representation for \( a(t, x) \) can be written as:
\begin{equation}\label{7}
    a(t, x) := a(t, x, u)=\frac{-\psi_t(t, x) + \Delta_x \psi(t, x)+(f(t, x, \cdot), \omega) + (u, \omega'')}{\psi(t, x)}.
\end{equation}

Next, we will investigate the following weak formulation of the inverse problem \eqref{1}–\eqref{overdetermination}.

\begin{опр}\label{opr1} 
Find a pair of functions \( \{ u(t, x, y), a(t, x) \} \), where \( a(t, x) \) has the form \eqref{7}, and the function \( u(t, x, y) \) satisfies the following conditions:
\begin{enumerate}
    \item \( u \in C([0, T]; L_2(D)) \), \,\, \( u \in L_2(0, T; \dot{W}_2^1(D)) \);
    \item \( D_t u \in L_2(Q) \);
    \item \( u(0, x, y) = \varphi(x, y) \) a.e.\ in \( D \);
    \item For any \( v \in \dot{W}_2^1(D) \) and almost every \( t \in (0, T] \), the following equality holds:
\end{enumerate}
\begin{equation}\label{equation}
    \int_D u_t v \, dx \, dy + \int_D (\nabla u \nabla v +  u_y  v_y) \, dx \, dy + \int_D a u  v \, dx \, dy  =  \int_D f v \, dx \, dy.
\end{equation}
\end{опр}

Before formulating the main results of this work let us introduce some notations (see Remark \ref{norms}):
\[
\tau_1=\frac{1+\varepsilon}{4},\,\, \tau_2=\frac{3+\varepsilon}{4},
\]
\[
A_\varepsilon = \sqrt{\pi}\left(1+\frac{1}{2\varepsilon}\right)^{\frac{1}{2}} ||\omega''|| \max_{t, x} \frac{1}{|\psi(t, x)|}, \,\, B= 2 A_\varepsilon^2 C^2_S,
\]
\[
\Psi (t, x) = \frac{-\psi_t(t, x) +\Delta_x \psi(t, x)+ (f(t, x, \cdot), \omega)}{\psi (t,x)} , \,\, \Psi_M = \max_{t, x} |\Psi(t, x)|,
\]
\[
R = ||\varphi||_{1, \tau_1} +  8 \Psi^2_M T ||\varphi||_{0, \tau_1} + ||\varphi||_{0, \tau_2} + 4 ||f||_{G_T, \tau_1},
 \]
 \[
R_1 = ||\varphi||_{1, \tau_1} +   ||\varphi||_{0, \tau_2} + 4 ||f||_{G_T, \tau_1},
\]
assuming that all given quantities exist. Here $\varepsilon$ is an arbitrary positive constant and $C_S$ is the constant from (\ref{Sobolev}).

The following theorem defines the conditions for local solvability of the problem posed.

\begin{thm}\label{theorem local}

Let $\varepsilon$ be any positive number and suppose the following conditions hold:
\begin{enumerate}
    \item \( f(t, x, y) \in  C (\overline{G_T}; \dot{W}_2^{\tau_1} (0, \pi))\);
    \item \( D_t \psi, \,\Delta \psi \in C( \overline{G_T}) \),\,\,\( \psi(t, x) \neq 0 \) for all \( (t, x) \in [0, T] \times G \);
    \item \( \varphi \in W_2^1(G; \dot{W}_2^{\tau_1}(0, \pi)) \) and \( \varphi \in L_2(G; \dot{W}_2^{\tau_2}(0, \pi)) \)
    \item   \( 2 \Psi_M T \leq A_\varepsilon C_S\);
    \item \(4 R B < 1 \).
      \end{enumerate}
Then, the inverse problem has a unique weak solution. Moreover, the following estimates hold:
\begin{equation}\label{uestimate1}
||u||^2_{L_2(Q)}  \leq K_1,
\end{equation}
 \begin{equation}\label{u_xu_y}
||\nabla u||^2_{G_T, \tau_1} +  ||u_t||^2_{G_T, \tau_1} + ||u||^2_{G_T, \tau_2} \leq K_2,
\end{equation}
\begin{equation}\label{estimate_a} 
||a||^2_{L_2(G_T)} \leq K_3,
\end{equation}
where $K_j,\, j=1, 2, 3,$ are positive constants depending on the data of the problem.

\end{thm}

The fourth condition of the theorem states the existence of a solution only for sufficiently small $T$.

Now we present a theorem on the global solvability of the problem under consideration.
\begin{thm}\label{theorem global}

Let all the conditions of Theorem \ref{theorem local} be satisfied except (4) and (5). Let the following two conditions be satisfied instead of these conditions: \begin{enumerate}
    \item \( 2 \Psi_M^2 C_P \leq A_\varepsilon^2 C^2_S\);
    \item \(4 R_1 B <1\).
      \end{enumerate}
Then all the statements of Theorem \ref{theorem local} remain valid.
\end{thm}

As can be seen from the conditions of this theorem, there are no restrictions on the finite time $T$. But the first condition of the theorem requires the smallness of the value $C_P$. As follows from Remark \ref{optimalC}, this condition means the smallness of the measure $G$.

\begin{зам}\label{zam.1} By virtue of S. M. Bernstein's theorem on the uniform convergence of trigonometric series (see \cite{Zyg}, p. 384), the conditions of the theorem on the functions $f$ and  $\varphi$ guarantee absolute and uniform convergence on \( [0, \pi] \) of the corresponding Fourier series for almost all $x\in G$ and $t\in (0, T)$.
\end{зам}

\section{A priori estimates}

Let us write the function $a$ in a form convenient for us.  To do this, we represent the solution to the forward problem (\ref{1}) as a formal Fourier series:
\begin{equation}\label{6}
    u(t, x, y) = \sum_{k=1}^\infty u_k(t, x) \sin k y,
\end{equation}
where \( u_k(t, x) \) are the unknown functions and \( \sin k y \) are the eigenfunctions of operator $L_0$.
Then formula \eqref{7} implies
\begin{equation}\label{h_final}
    a(t, x) = \frac{-\psi_t (t, x) +\Delta_x \psi(t, x)+ (f(t, x, \cdot), \omega) + \sum_{j=1}^\infty (\sin \sqrt{\lambda_j} y , \omega'') u_j(t,x) }{\psi (t,x)}.
\end{equation}

Now we decompose the functions \( f(t, x, y) \),  and \( \varphi(x, y) \) into Fourier series (see Remark \ref{zam.1}) and denote their corresponding Fourier coefficients by \( f_k(t, x) \),  and \( \varphi_k(x) \). Substitute these series and the representation \eqref{6} for \( u(t, x, y) \) into equation \eqref{equation}, with the test function \( v(x, y) \) replaced by \( w(x) \sin k y \), where \( w \in \dot{W}_2^1(G) \). 
Thus we obtain the following equation to determine the unknown coefficients \( u_k(t, x) \) (see \eqref{6}):
\begin{equation}\label{81}
\int_G D_t u_k w \, dx + \int_G \nabla u_k \nabla w \, dx + \lambda_k \int_G u_k w \, dx = \int_G f_k w \, dx - \int_G a u_k w \, dx.
\end{equation}
Note that here, when calculating the third term on the left-hand side, we used the equalities $v_y= \sqrt{\lambda_k} w(x) \cos ky$ and $u_y=\sum_{j=1}^\infty \sqrt{\lambda_j} \, u_j(t,x) \cos \sqrt{\lambda_j} y$ and then used the orthogonality  of the system $\{\cos \sqrt{\lambda_j} y\}$ in $L_2(0, \pi)$.

Our next task is to find the unknown coefficients $u_k$ from the equations (\ref{81}). But here \( a \) is defined using all \( u_j(t, x) \), \( j = 1, 2, \dots \) (see \eqref{h_final}). This fact creates a certain problem and in order to solve it we use the method of successive approximations, i.e. using recurrence relations we construct a sequence \( \{ u_k^i \} \), \( i = 1, 2, \dots \), and then we prove that \( \{ u_k^i \} \to \{ u_k \} \) as \( i \to \infty \) in the appropriate norm.

The corresponding recurrence relations for all \( w \in \dot{W}_2^1(G) \), \( k \geq 1 \) and \( i \geq 1 \) have the form:
\begin{equation}\label{9}
\int_G D_t u_k^i w \, dx + \int_G \nabla u_k^i \nabla w \, dx + \lambda_k \int_G u_k^i w \, dx 
\end{equation}
\[
= \int_G f_k w \, dx - \int_G \Psi(t, x) u_k^{i-1}  w \, dx - \int_G \frac{ \sum_{j=1}^\infty (\sin \sqrt{\lambda_j} y , \omega'') u^{i-1}_j(t,x) }{\psi (t,x)}  u_k^{i-1} w \, dx,
\]
with the initial conditions:
\begin{equation}\label{10}
u_k^i(0, x) = \varphi_k(x),
\end{equation}
where $u^0_k\equiv 0$ for all $k$ and
\[
\Psi (t, x) = \frac{-D_t \psi(t, x) +\Delta_x \psi(t, x)+ (f(t, x, \cdot), \omega)}{\psi (t,x)},
\]

If all $u^{i-1}_j$, $j= 1, 2,  \cdots, $  are known, then the problem \eqref{9}–\eqref{10} is well studied (see, e.g., \cite{Ladijenskaya}, Chapter 3,  \cite{AMux2}, \cite{AMux}).
.  It is not hard to show (see, e.g., \cite{Ladijenskaya}, Chapter 3) that if \(f_k\in L_2(G_T)\) and the initial function \(\varphi_k(x) \in L_2(G)\) (which holds in our case), then there exists a unique strong solution of the problem satisfying \(u_k^i(t, x) \in L_2(G_T) \cap L_2(0, T; W_2^2(G) \cap \dot{W}_2^1(G))\) and \(D_t u_k^i(t, x) \in  L_2(G_T)\). Evidently, such a strong solution is also a weak solution.

\begin{лем}\label{main} Let $2 \Psi_M T \leq A_\varepsilon C_S$. Then we have
\[
\sum_{k=1}^\infty \sqrt{\lambda_k}^{1+ \varepsilon} \left [\| D_t u_k^i \|_{L_2(G_T)}^2   + \sup_t||\nabla u_k^i||_{L_2(G)}^2 + \lambda_k \sup_t||u_k^i||_{L_2(G)}^2\right]
\]
\[
\leq  \sum_{k=1}^\infty \sqrt{\lambda_k}^{1+ \varepsilon}\left[||\nabla \varphi_k||_{L_2(G)}^2 + 8 \Psi^2_M T \| \varphi_k \|_{L_2(G)}^2\right] + \sum_{k=1}^\infty \sqrt{\lambda_k}^{2+ \varepsilon} ||\varphi_k||_{L_2(G)}^2
\]
\[
 + 4 \sum_{k=1}^\infty \sqrt{\lambda_k}^{1+ \varepsilon} ||f_k||^2_{L_2(Q_t)} +2 A_\varepsilon^2 C^2_S  \left( \sum_{k=1}^\infty \sqrt{\lambda_k}^{1+ \varepsilon} ||u_k^{i-1}||_{W^1_2(G_T)}^2\right)^2.
\]

\end{лем}
\begin{proof}
Substitute \( w = D_t u_k^i \) into equation \eqref{9} and integrate it, to get
\begin{equation}\label{D+nabla}
\int_0^t\left[\| D_t u_k^i \|_{L_2(G)}^2 + \int_G \nabla u_k^i\, D_t\, \nabla u_k^i \, dx + \lambda_k \int_G u_k^i\, D_t\,u_k^i \, dx \right] ds
\end{equation}
$$
=\int_0^t\left[ \int_G f_k  D_t u_k^i \, dx - \int_G \Psi(t, x) u_k^{i-1} D_t u_k^i \, dx\right] ds -$$ 
$$-\int_0^t\left[\int_G \frac{ \sum_{j=1}^\infty  (\sin \sqrt{\lambda_j} y , \omega'') u^{i-1}_j(t,x) }{\psi (t,x)}  u_k^{i-1} D_t u_k^i \, dx\right] ds.
$$
Let us denote the three integrals on the right-hand side by $J^j_k$, $j=1, 2, 3$, respectively. First note that
\[
2 \int_0^t D_s u_k^i  u_k^i \,   ds = |u_k^i(t, x, y)|^2 - |u_k^i(0, x, y)|^2.
\]
A similar equality holds for the gradient of the function $u_k^i$. Then equality (\ref{D+nabla}) can be rewritten as the inequality
\begin{equation}\label{Dt}
\| D_t u_k^i \|_{L_2(G_t)}^2 + \frac{1}{2}||\nabla u_k^i||_{L_2(G)}^2 + \frac{\lambda_k}{2} ||u_k^i||_{L_2(G)}^2 \leq 
\end{equation}
$$
\leq \frac{1}{2}||\nabla \varphi_k||_{L_2(G)}^2 + \frac{\lambda_k}{2} ||\varphi_k||_{L_2(G)}^2 +\int_0^t \sum_{j=1}^3 |J^j_k|ds.
$$
Apply the Cauchy inequality (\ref{Cauchy inequality}) with $\delta =4$ to get 
\[
|J_k^1| \leq 2||f_k||^2_{L_2(G_t)}  + \frac{1}{8} ||D_tu_k^i||^2_{L_2(G_t)},
\]
\[
|J_k^2| \leq 2 \Psi^2_M ||u^{i-1}_k||^2_{L_2(G_t)} + \frac{1}{8}   \, ||D_tu_k^i||^2_{L_2(G_t)}.
\]
Let us estimate the integral $J^3_k$. The Cauchy-Schwarz inequality implies
\[
|\sum_{j=1}^\infty (\sin \sqrt{\lambda_j} y , \omega'') u_j^{i-1} (t, x)| \leq \sqrt{\pi}||\omega''|| \left (\sum_{j=1}^\infty \sqrt{\lambda_j}^{-(1+ \varepsilon)}\right)^{\frac{1}{2}} \left (\sum_{j=1}^\infty \sqrt{\lambda_j}^{1+ \varepsilon} |u_j^{i-1}|^2\right)^{\frac{1}{2}},
\]
we have
\[
J^3_k \leq  A_\varepsilon \int_{G_t} \left (\sum_{j=1}^\infty \sqrt{\lambda_j}^{1+ \varepsilon} |u_j^{i-1}|^2\right)^{\frac{1}{2}}  |u_k^{i-1}| |D_t u_k^i| \, dx ds.
\]
We apply the Cauchy inequality (\ref{Cauchy inequality}) with the constant $\delta =2$
to the last integral:
\[
J^3_k  \leq A_\varepsilon^2\int_{G_t}  \sum_{j=1}^\infty \sqrt{\lambda_j}^{1+ \varepsilon} |u_j^{i-1}|^2 |u_k^{i-1}|^2 \, dx ds  +\frac{1}{4} ||D_t u_k^i||_{L_2(G_t)}^2.
\]
By virtue of the Cauchy-Schwarz and Sobolev (see (\ref{Sobolev}), and Remark \ref{m})  inequalities, we have 
\[
\int_0^t\int_G |u_j^{i-1}|^2 |u_k^{i-1}|^2 \, dx ds \leq ||u_j^{i-1}||_{L_4(G_t)}^2 ||u_j^{i-1}||_{L_4(G_t)}^2 \leq C^2_S || u_j^{i-1}||_{W^1_2(G_t)}^2 ||u_k^{i-1}||_{W^1_2(G_t)}^2.
\]
Thus
\[
\sum_{k=1}^\infty \sqrt{\lambda_k}^{1+ \varepsilon} |J_k^3| \leq 
A_\varepsilon^2 C^2_S  \left( \sum_{k=1}^\infty \sqrt{\lambda_k}^{1+ \varepsilon} ||u_k^{i-1}||_{W^1_2(G_t)}^2\right)^2 + \frac{1}{4} \sum_{k=1}^\infty \sqrt{\lambda_k}^{1+ \varepsilon}||D_tu_k^i||_{L_2(G_t)}^2.
\]
Substituting the estimates of $J_k^j$ into (\ref{Dt}) we get
\begin{equation}\label{main part}
\frac{1}{2}\sum_{k=1}^\infty \sqrt{\lambda_k}^{1+ \varepsilon} \left [\| D_t u_k^i \|_{L_2(G_t)}^2 ds  + ||\nabla u_k^i||_{L_2(G)}^2 + \lambda_k ||u_k^i||_{L_2(G)}^2\right]
\end{equation}
\[
\leq  \frac{1}{2}\sum_{k=1}^\infty \sqrt{\lambda_k}^{1+ \varepsilon}||\nabla \varphi_k||_{L_2(G)}^2 + \frac{1}{2}\sum_{k=1}^\infty \sqrt{\lambda_k}^{2+ \varepsilon} ||\varphi_k||_{L_2(G)}^2 + 2 \sum_{k=1}^\infty \sqrt{\lambda_k}^{1+ \varepsilon} ||f_k||^2_{L_2(Q_t)}
\]
\[
+2 \Psi^2_M \sum_1^\infty \sqrt{\lambda_k}^{1+\varepsilon}   ||u_k^{i-1}||_{L_2(G_t)}^2  +  A_\varepsilon^2 C^2_S  \left( \sum_{k=1}^\infty \sqrt{\lambda_k}^{1+ \varepsilon} ||u_k^{i-1}||_{W^1_2(G_t)}^2\right)^2.
\]
Applying Lemma \ref{Poincaré} to function $g(t) = ||u_k^{i-1} ||_{L_2(G)}$, we get
\begin{equation}\label{Lemma 1}
||u_k^{i-1}||_{L_2(G_T)}^2 \leq T^2 ||\frac{\partial}{\partial t} u_k^{i-1}||^2_{L_2(G_T)} + 2 T \| \varphi_k \|_{L_2(G)}^2.
\end{equation}
Therefore, taking the maximum over $t\in [0, T]$ in the previous inequality and then using the condition of the lemma $2 \Psi_M T \leq A_\varepsilon C_S$, we will have the statement of the lemma.
\end{proof}

\begin{зам}The meaning of using the estimate (\ref{Lemma 1}) is as follows: we need the coefficients before $u_k^{i-1}$ to be small. However, the number $\Psi_M$ cannot be made small by selecting $\psi$ and $f$. After using the estimate (\ref{Lemma 1}), the condition $2 \Psi_M T \leq A_\varepsilon C_S$ can be achieved by taking $T$ small enough.
\end{зам}

\begin{лем}\label{Final1} Let $R$ and $B$ be the quantities defined above. Then if $T \leq 1$ and
$$
2 \Psi_M T \leq A_\varepsilon C_S, \,\, q:=4 R B < 1,
   $$
then the following estimate is valid for all $i$
\begin{equation}\label{final}
F( u^i):=\sum_{k=1}^\infty \sqrt{\lambda_k}^{1+ \varepsilon} \left [\| D_t u_k^i \|_{L_2(G_T)}^2   + \sup_t||\nabla u_k^i||_{L_2(G)}^2 + \lambda_k \sup_t||u_k^i||_{L_2(G)}^2\right] \leq 2 R.
\end{equation}

\end{лем}
\begin{proof} Lemma \ref{main} implies
\[
F( u^i) \leq R + B \left( \sum_{k=1}^\infty \sqrt{\lambda_k}^{1+ \varepsilon} ||u_k^{i-1}||_{W^1_2(G_T)}^2\right)^2.
\]
    Let $u_k^{0} =0$. Then
    $F( u^1) \leq R < 2 R.$ Therefore, by the condition of the theorem and taking into account the condition $\lambda_k \geq 1$,
    $$
    F( u^2)  \leq 2 R +B \left( \sum_{k=1}^\infty \sqrt{\lambda_k}^{1+ \varepsilon} \left[\| D_t u_k^1 \|_{L_2(G_T)}^2+ T \sup_t||\nabla u_k^1||_{L_2(G)}^2 + T\sup_t|| u_k^1||_{L_2(G)}^2 \right]\right)^2
    $$ 
    \[
    \leq R + 4 R^2 B\leq 2 R.
    \]
    Continuing this process, we obtain the statement of the theorem.
\end{proof}

Now in the relation (\ref{main part}) instead of the estimate (\ref{Lemma 1}) we will use the estimate (see (\ref{DirixlePoincaré}))
\begin{equation}\label{estimate and remark}
  ||u_k^{i-1}||_{L_2(G_T)}^2 \leq C_P ||\nabla u_k^{i-1}||_{L_2(G_T)}^2.
\end{equation}
Using this estimate, the smallness of the coefficient in front of $u_k^{i-1}$ can be achieved due to the smallness of the measure of the region $G$ (see Remark \ref{optimalC}).

If $R_1$ is the quantity defined above then instead of Lemma \ref{Final1} we obtain the following statement
\begin{лем}\label{Final2} Let $2 \Psi_M^2 C_P \leq A_\varepsilon^2 C^2_S$ and $4 R_1 B <1$. Then $F(u^i) \leq 2 R_1$ for all $i$.
\end{лем}
\section{Convergence}
We aim to demonstrate that the sequences \( u_k^i \), \( \nabla u_k^i \), and \( D_t u_k^i \), for \( i=1,2,\dots \), are fundamental in $L_2(G_T)$ norm for all \( k \geq 1 \).

Further, we will assume that the conditions of Lemma \ref{Final1} are fulfilled.

By reiterating the previous arguments, we can establish estimates \eqref{final} for \( u_k^{i+p} - u_k^i \), where \( p = 1,2,3,\dots \). Indeed, consider writing equality (\ref{9}) for \( u_k^{i+1} \) and subtracting (\ref{9}) from it. This  leads  in (noting that here \( f_k(t, x) \equiv 0 \) and $\Psi (t, x) \equiv 0$):
$$
\int_G D_t (u_k^{i+1}-u_k^i)\, w \, dx + \int_G \nabla(u_k^{i+1}-u_k^i) \,\nabla w \, dx + \lambda_k \int_G (u_k^{i+1}-u_k^i) \,w \, dx 
$$
\[
=  - \int_G \frac{ \sum_{j=1}^\infty (\sin \sqrt{\lambda_j} y , \omega'') ( u^i_j(t,x)- u^{i-1}_j(t,x)) }{\psi (t,x)}   u_k^{i-1} w \, dx
\]
\[
- \int_G \frac{ \sum_{j=1}^\infty (\sin \sqrt{\lambda_j} y , \omega'')  u^i_j(t,x) }{\psi (t,x)}  (u_k^{i} - u_k^{i-1}) w \, dx,
\]
for all \( w \in \dot{W}_2^1(0, 1) \), with the initial condition:
\begin{equation}\label{101}
(u_k^{i+1}-u_k^i)(0, x) = 0.
\end{equation}
By applying the same  reasoning as in the proof of Lemma \ref{Final1}, but substituting \eqref{101} for the initial condition \eqref{10}, we obtain the inequality:
\[
F(u^{i+1} - u^{i})  \leq B \sum_{k=1}^\infty \sqrt{\lambda_k}^{1+ \varepsilon} ||u_k^{i}-u_k^{i-1}||_{W^1_2(G_T)}^2 \sum_{k=1}^\infty \sqrt{\lambda_k}^{1+ \varepsilon}(||u_k^{i}||_{W^1_2(G_T)}^2 + ||u_k^{i-1}||_{W^1_2(G_T)}^2).
\]
Since \( u_k^0 = 0 \) for all \( k \geq 1 \), then from Lemma \ref{Final1}, we get
\[
F(u^{2} - u^{1}) \leq B (2 R)^2 < R q. 
\]
Likewise
\[
F(u^{3} - u^{2}) \leq  B\,R \,q\, 4R \leq R q^2.
\]
For any \( i \geq 2 \) the last two estimates imply:
\[
F(u^{i+1} - u^{i})  \leq  R q^i.
\]
Hence, if we write $u_k^{i+p} - u_k^i=u_k^{i+p} - u_k^{i+p-1}+ u_k^{i+p-1}- u_k^{i+p-2}+\cdots$, then
\begin{equation}\label{fundamental}
F(u^{i+p} - u^{i}) \leq R  q^{i-1} ( 1 + q +q^2 + \cdots q^p)= R  q^{i-1} \frac{1-q^{p-1}}{1-q}.
\end{equation}
This indicates that the sequence \(\{F(u^{i})\}\) is fundamental in \(L_2(G)\). Therefore, for each \(k \geq 1\), there exist functions \(u_k(t,x) \in L_2(G_T)\) such that, as \(i \to \infty\), we have:
\begin{equation}\label{converU}
\begin{cases}
u_k^i(t, x) \to u_k (t, x)\quad \text{in } L_2(G_T), \quad k=1,2, 
\cdots,\\
\sum_{k=1}^\infty \sqrt{\lambda_k}^{1+ \varepsilon}\left(||\nabla (u_k^i - u_k (t, x)) ||_{L_2(G_T)}^2 + ||D_t (u_k^i - u_k (t, x)) \|_{L_2(G_T)}^2 \right) \to 0, \\
\sum_{k=1}^\infty \sqrt{\lambda_k}^{3+ \varepsilon} ||u_k^i -u_k (t, x)||_{L_2(G_T)}^2  \to 0,
\end{cases}
\end{equation}
From the reasoning given in the proof of (\ref{fundamental}), it also follows (see definition of $a(t,x, u)$ in (\ref{7}))
\begin{equation}\label{conver_au}
||a(t, x, \sum_j u_j^{i+p}) u_k^{i+p} - a(t, x, \sum_j u_j^{i}) u_k^{i}||_{L_2(D)}^2\to 0,\,\, \text{as} \,\, i\to \infty.
\end{equation}
If we take into account the convergence (\ref{converU}) and Parseval's equality, we come to the conclusion that the function
\[
u (t, x, y) = \sum\limits_{k=1}^{\infty} u_k (t, x) \sin k y,
\]
is well-defined element of $L_2(Q)$ with properties
\begin{equation}\label{Du_estimate}
||\nabla u||^2_{G_T, \tau_1} +  ||u_t||^2_{G_T, \tau_1} + ||u||^2_{G_T, \tau_2} \leq K_2,
\end{equation}
From here and from the Poincaré inequality, we obtain
\begin{equation}\label{uestimate}
||u||^2_{L_2(Q)}  \leq K_1.
\end{equation}
Again from (\ref{converU})  and Parseval's equality, it follows that the function $a(t,x)$ defined by the formula (\ref{h_final}) is a well-defined element of the space $L_2(G_T)$:
\begin{equation}\label{a}
||a||^2_{L_2(G_T)}  \leq K_3.
\end{equation}

Obviously, from the estimates (\ref{Du_estimate}) - (\ref{a}) the estimates (\ref{uestimate1}) - (\ref{estimate_a}) follow.

\begin{зам}All the above results of this section are also valid under the condition of Lemma \ref{Final2}.
\end{зам}

\section{Proof of Theorems}

To the proof of Theorem \ref{theorem local}, we first show the existence of a weak solution of the inverse problem. We integrate \eqref{9} over $t\in [0,T]$, and then pass to the limit $i \rightarrow \infty$ and, taking into account \eqref{converU} and (\ref{conver_au}), we obtain the following equalities, valid for an arbitrary \( w \in \dot{W}_2^1(G) \) and for all $k$: \begin{equation}\label{22}
\int_0^t\left[( D_t^{\alpha} u_k, w ) + (\nabla u_k, \nabla w ) + \lambda_k ( u_k, w )+ (a u_k, w)\right] d\tau =\int_0^t ( f_k, w ) d\tau.  
\end{equation}
Next, we take the function $w=v(x,y) \sin k y$, $v\in \dot{W}_2^1(\Omega)$, and integrate equality (\ref{22}) over $y\in [0, \pi]$. Then, summing over $k$ from 1 to $\infty$, we obtain for all $t\in [0, T]$ 
 \begin{equation}\label{solution_t}
\int_0^t \int_D \left[D_t^\alpha u v + \nabla u \nabla v + u_{y} v_{y} +au v \right]\, dx \, dy d\tau = \int_0^t\int_D f v\, dx \, dy d\tau.
\end{equation}

Note that if $s(t)$ is integrable in any subset $(0,t)$ of the interval $[0,T]$ and $\int_0^t s(\tau) d\tau =0$, then obviously $s(t)=0$ almost everywhere on $[0,T]$. Therefore, the equality (\ref{solution_t}) coincides with (\ref{equation}). Consequently, the function $u(t,x,y)$ defined by formula (\ref{6}), together with $a$, defined by formula (\ref{h_final}), is a weak solution of the inverse problem, i.e., it satisfies all the conditions of Definition \ref{opr1}. As for the estimates \eqref{uestimate1} - \eqref{estimate_a}, they are proved above (see the estimates \eqref{Du_estimate} - (\ref{a})).  

Let us prove the uniqueness of the weak solution of the inverse problem.

Suppose the opposite, i.e. there are two solutions to the inverse problem: \( (u^1, a(t, x, u^1)) \) and \( (u^2, a(t, x, u^2)) \), where $a(t, x, u^k)$ is defined by (\ref{7}) (note, that  \( f(t,x,y) \equiv 0 \) and $\varphi(x, y) \equiv 0$). Let us denote $u= u^1- u^2$ and
\[
u_j(t, x) = \int\limits_\Omega u(t, x, y) \sin \sqrt{\lambda_j} y \, dy, \,\, u^k_j(t, x) = \int\limits_\Omega u^k(t, x, y) \sin \sqrt{\lambda_j} y \, dy.
\]
Consider the series defining the function $a(t, x, u^k)$ as given in  (\ref{h_final}):
\[\sum_{j=1}^\infty (\sin \sqrt{\lambda_j} y , \omega'') u^k_j (t, x), \, \, k=1, \, 2.
\]
Estimate (\ref{a}) means that this function
exists almost everywhere in $G_T$. Therefore, the function
\[
 a (t, x, u^k) = \frac{-\psi_t (t, x) +\Delta_x \psi(t, x) + \sum_{j=1}^\infty (\sin \sqrt{\lambda_j} y , \omega'') u^k_j(t,x) }{\psi (t,x)}, \, \, k=1, \, 2.
\]
(see (\ref{h_final})) is correctly defined. Then for functions $u^k$ one has (see the estimate of $F(u^{i+1} - u^{i})$):
\[
F(u)  \leq B \sum_{j=1}^\infty \sqrt{\lambda_j}^{1+ \varepsilon} ||u_j^1-u_j^2||_{W^1_2(G_T)}^2 \sum_{j=1}^\infty \sqrt{\lambda_j}^{1+ \varepsilon}(||u_j^1||_{W^1_2(G_T)}^2 + ||u_j^2||_{W^1_2(G_T)}^2).
\]
Using iteration, we get:
\[
F(u)  \leq  R q^n, \,\, n =1, 2, \cdots .
\]
Hence,
\[
u_j (t,x) = u^1_j (t,x) - u^2_j (t,x)=0 \quad \text{a.e. in } G_T.
\]
 Therefore,  $a_1 (t, x, u^1)\equiv a_2 (t, x, u^2)$ \text{a.e. in} $G_T$. Completeness of the system $\{\sin \sqrt{\lambda_j} y\}$ implies
\[
u(t,x,y) = 0 \quad \text{for almost all}\,\,  (t, x, y) \in Q.
\]
Thus, Theorem \ref{theorem local} is completely proved.

Using Lemma \ref{Final2}, Theorem \ref{theorem global} is proved by repeating similar arguments.

\section{Strong solution}

In this section, we study the existence and uniqueness of strong generalized solutions of the inverse problem \eqref{1}–\eqref{overdetermination}. The strong solution is formulated as
follows:
\begin{опр} Weak solution of the inverse problem \eqref{1}–\eqref{overdetermination} is a strong if
\begin{enumerate}
    \item \( u\in L_2 (0, T; \dot{W}_2^1(D) \cap W_2^2(D))\);
    \item \( D_t u \in L_2(Q)\);
    \item \( a\in L_2(G_T) \)
    \item  Every equation in (\ref{1}) and \eqref{overdetermination} holds a.e. in their corresponding domain.
      \end{enumerate}
\end{опр}
The theorem on local strong solvability has the form
\begin{thm}\label{theorem strong}

Let $\varepsilon=5$ and all the conditions of Theorem \ref{theorem local} are satisfied.
Then, the inverse problem has a unique strong solution. Moreover, the following estimates hold:
$$
||u||^2_{L_2(Q)} + ||\Delta u||^2_{L_2(Q)} +  ||u_t||^2_{L_2(Q)} + ||u_{yy}||^2_{L_2(Q)} \leq K_6,
$$
$$
||a||^2_{L_2(G_T)} \leq K_7,
$$
where $K_j,\, j=6, 7$ are positive constants depending on the data of the problem.

\end{thm}
\begin{proof}
Let us take a nonnegative function  $\chi(x) \in C_0^\infty (G)$, $\chi \leq 1,$ and substitute \( w = \chi \Delta u_k^i \) into equation \eqref{9}, and then integrate over $t$ from $0$ to $t$, to obtain:
\begin{equation}\label{delta_start}
\int_0^t\int_G D_s u_k^i  \chi \Delta u_k^i \, dx  ds + \int_0^t \int_G \left(\nabla u_k^i \nabla (\chi \Delta u_k^i) + \lambda_k u_k^i \chi \Delta u_k^i \right)dx ds 
\end{equation}
\[
=\int_0^t\left[ \int_G f_k \chi \Delta u_k^i \, dx - \int_G \Psi(t, x) u_k^i \chi \Delta u_k^i \, dx - \int_G \frac{ \sum_{j=1}^\infty  (\sin \sqrt{\lambda_j} y , \omega') u^{i-1}_j(t,x) }{\psi (t,x)}  u_k^{i-1} \chi \Delta u_k^i \, dx\right] ds.
\]
One has
\[
\int_0^t \int_G \nabla u_k^i \nabla (\chi \Delta u_k^i )dx ds = - \int_0^t \int_G  \chi |\Delta u_k^i|^2 dx ds.
\]
Now on the left in (\ref{delta_start}) we leave only this term and we transfer everything remaining to the right side and we denote the five integrals that arise on the right by $I_k^j$, $j= 1, \cdots, 5,$ accordingly. So,
\begin{equation}\label{sum}
\int_0^t \int_G  \chi |\Delta u_k^i|^2 dx ds = \sum_{j=1}^5 I_k^j.
\end{equation}
By virtue of the Cauchy inequality (\ref{Cauchy inequality}) with the constant $\delta =5$, we have 
\[
|I_k^1|:= \left|-\int_0^t\int_G  D_s u_k^i  \chi \Delta u_k^i \, dx  ds\right| \leq \frac{5}{2} ||D_t u_k^i||_{L_2(Q_T)}^2 + \frac{1}{10} \int_0^t \int_G  \chi |\Delta u_k^i|^2 dx ds,
\]
here we applied the inequality $\chi^2 \leq \chi$. In a similar way
\[
|I_k^2|:= \left|- \lambda_k\int_0^t \int_G  u_k^i \chi \Delta u_k^idx ds\right| \leq \frac{5}{2} ||\lambda_k u_k^i||_{L_2(Q_T)}^2 + \frac{1}{10} \int_0^t \int_G  \chi |\Delta u_k^i|^2 dx ds.
\]
We also have
\[
|I_k^3|:= \left|\int_0^t \int_G f_k  \chi \Delta u_k^idx ds\right| \leq \frac{5}{2} ||f_k||_{L_2(Q_T)}^2 + \frac{1}{10} \int_0^t \int_G  \chi |\Delta u_k^i|^2 dx ds,
\]
\[
|I_k^4|:= \left|- \int_0^t \int_G \Psi(t, x) u_k^i \chi \Delta u_k^i \,dx ds\right| \leq \frac{5}{2} \Psi_M^2 ||u^i_k||_{L_2(Q_T)}^2 + \frac{1}{10} \int_0^t \int_G  \chi |\Delta u_k^i|^2 dx ds,
\]
For the last integral $I_k^5$, using similar reasoning as for the estimate $J_k^3$, we will obtain
\[
I_k^5\leq A_\varepsilon^2 C^2_S  \sum_{j=1}^\infty \sqrt{\lambda_j}^{1+ \varepsilon} \sqrt{\lambda_
k}^{1+ \varepsilon}|| u_j^{i-1}||_{W^1_2(G_t)}^2 ||u_k^{i-1}||_{W^1_2(G_t)}^2  +\frac{1}{10} \int_0^t \int_G  \chi |\Delta u_k^i|^2 dx ds,
\]
note that adding $\sqrt{\lambda_
k}^{1+ \varepsilon}$ strengthens the inequality.

Now we sum the equality (\ref{sum}) over all $k$ and use the estimates of $I_k^j$. Then we will get
\[
\sum_{k=1}^\infty\int_0^t \int_G  \chi |\Delta u_k^i|^2 dx ds 
\]
\[
\leq 5\sum_{k=1}^\infty\left[||D_t u_k^i||_{L_2(Q_T)}^2 + ||\lambda_k u_k^i||_{L_2(Q_T)}^2+ ||f_k||_{L_2(Q_T)}^2 +\Psi_M^2 ||u^i_k||_{L_2(Q_T)}^2 \right]
\]
\[
 +  A_\varepsilon^2 C^2_S  \left( \sum_{k=1}^\infty \sqrt{\lambda_k}^{1+ \varepsilon} ||u_k^{i-1}||_{W^1_2(G_t)}^2\right)^2.
\]
As shown above, the right-hand side is bounded by a positive constant $K_4$, depending on the data of the problem. And from the convergence of the right-hand side at $i\to \infty$ follows the convergence of the left-hand side. As a result, we obtain the estimate
\[
\int_0^t \int_D  \chi |\Delta u|^2 dx dy ds \leq K_4.
\]
Since the constant $K_4$ does not depend on $\chi$, then by virtue of the theorem on monotone convergence (see. e.g., \cite{ReedSimon1978}, Theorem I.10) we obtain
\[
||\Delta u||^2_{L_2(Q)} \leq K_4.
\]

From estimate (\ref{u_xu_y}) it follows that $u_t\in L_2(Q_T)$ and if we put $\varepsilon=5$, then $u_{yy}\in L_2(Q_T)$. Consequently, for $\varepsilon=5$ we have
\[
||u_t||^2_{L_2(Q)} + ||u_{yy}||^2_{L_2(Q)} \leq K_5, 
\]
where $K_5$ is a positive constant, depending on the data of the problem.
\end{proof}

If $\varepsilon=5$ and the conditions of Theorem \ref{theorem global} are satisfied, then a similar theorem on global solvability holds.
\section{Conclusions}
 
In this paper, we study the inverse problem of potential recovery for diffusion equations. The elliptic part of the equation is given by $\Delta_x u + u_{yy}$, $x\in \mathbb{R}^m$, $y\in \mathbb{R}^1$, and the unknown function depends on both time and a part of the spatial variables, denoted as $h(t, x)$. It is possible to consider the case $y\in \mathbb{R}^n$, but then some properties of the eigenvalues of the Laplace operator with the Dirichlet condition are needed.

In this case, the overdetermination condition is given in integral form. To our knowledge, such a problem has not been studied for diffusion equations before. The existence and uniqueness of a weak local and global solution of the inverse problem, as well as coercive estimates, are established.

If the data of problem are smoother, then it is possible to prove the unique solvability of the inverse problem in the strong statement. In this case, sufficient conditions are found that guarantee the existence of both a local and a global solution.

If the data of problem are even smoother, then apparently a classical solution also exists. However, this will be the topic of the next article.

\

\begin{center}
ACKNOWLEDGEMENTS    
\end{center}

The authors are grateful to Sh. A. Alimov, and Z. Sabirov for discussions of
these results. The authors also acknowledge financial support from the Ministry of Innovative Development of the Republic of Uzbekistan, Grant No F-FA-2021-424.


\end{document}